\documentclass[8pt]{article}
\usepackage[utf8]{inputenc}

\usepackage{biblatex}
\usepackage{tikz}
\usepackage{graphicx}
\usepackage{caption}
\usepackage{subcaption}
\usepackage{lineno,hyperref}
\usepackage{amssymb, amsthm, amsmath}
\usepackage{mathtools}
\usepackage{graphics} 
\usepackage{dsfont}
\usepackage{amsthm}
\usepackage[english]{babel}
\usepackage{fancyhdr}
\usepackage{wrapfig}
\usepackage{color}
\usepackage[shortlabels]{enumitem}
\usepackage{hyperref}
\usepackage[utf8]{inputenc}
\usepackage{multirow}
\usepackage{comment}

\theoremstyle{plain}

\newtheorem{theorem}{Theorem}[section]
\newtheorem{lemma}{Lemma}[section]

\theoremstyle{definition}

\title{All Graphs with a Failed Zero-Forcing Number of Two}
\author{Luis Gomez, Karla Rubi, Jorden Terrazas}
\date{July 2021}

\usepackage{authblk}
\author[1]{Luis Gomez}
\author[2]{Karla Rubi}
\author[3]{Jorden Terrazas}
\author[4]{Darren Narayan}

\affil[1]{University of Arkansas}
\affil[2]{California State University - Dominguez Hills}
\affil[3]{Southern Methodist University}
\affil[4]{Rochester Institute of Technology}

\usepackage{biblatex} 
\begin{document}

\maketitle
\begin{abstract}
 Given a graph $G$, the zero-forcing number of $G$, $Z(G)$, is the smallest cardinality of any set $S$ of vertices on which repeated applications of the forcing rule results in all vertices being in $S$. The forcing rule is: if a vertex $v$ is in $S$, and exactly one neighbor $u$ of $v$ is not in $S$, then $u$ is added to $S$ in the next iteration. Zero-forcing numbers have attracted great interest over the past 15 years and have been well studied. In this paper we investigate the largest size of a set $S$ that does not force all of the vertices in a graph to be in $S$. This quantity is known as the failed zero-forcing number of a graphs and will be denoted by $F(G)$, and has received attention in recent years. We present new results involving this parameter. In particular, we completely characterize all graphs $G$ where $F(G)=2$, solving a problem posed in 2015 by Fetcie, Jacob, and Saavedra.   
\end{abstract}

\section{Introduction}
Given a graph $G$, the zero-forcing number of $G$, $Z(G)$, is the smallest cardinality of any set $S$ of vertices on which repeated applications of the forcing rule results in all vertices being in $S$. The forcing rule is: if a vertex $v$ is in $S$, and exactly one neighbor $u$ of $v$ is not in $S$, then $u$ is added to $S$ in the next iteration. Zero forcing numbers have attracted great interest over the past 15 years and have been well studied \cite{Fallat}. In this paper we investigate the largest size of a set $S$ that does not force all of the vertices in a graph to be in $S$. This quantity is known as the failed zero forcing number of a graphs and will be denoted by $F(G)$, and has received attention in recent years \cite{Fetcie}. Independently, a closely related property called the zero blocking number of a graph was introduced by Beaudouin-Lafona, Crawford, Chen, Karst, Nielsen, and Sakai Troxell \cite{Karst} and Karst, Shen, and Vu \cite{KarstB}. The zero blocking number of a graph $G$ equals $|V(G)|-F(G)$.

We will use $K_n$, $P_n$, $C_n$, to denote the complete graph, path, and cycle on $n$ vertices, respectively. The complete bipartite graph with $r$ vertices in one part and $t$ in the other part will be denoted $K_{r,t}$. The disjoint union of graphs $G$ and $H$ will be denoted by $G+H$. Following the terminology of \cite{Erdos}, we will refer an induced $K_{1,2}$ as a \textquotedblleft cherry\textquotedblright. A vertex $v$ is a cut-vertex of a graph $G$ if $G-v$ has more components than $G$. We will refer to a subgraph that contains a cut-vertex $v$ of degree 3 that is adjacent to vertices $u$ and $w$, which are adjacent to each other and have degree 2 as a pendant triangle.

Throughout the paper we will use colorings to describe our failed zero-forcing sets where the vertex in $S$ will be colored blue (or referred to as filled) and the vertices not in $S$ will be colored white (or referred to as uncolored). A vertex will be called \textit{true blue} if it is colored blue and all of its neighbors are colored blue. A vertex that is not a true blue vertex will be referred to as \textit{non true blue}.

Failed zero-forcing numbers were introduced by Fetcie, Jacob, and Saavedra \cite{Fetcie} where they established numerous results. In particular they proved that for any connected graph $0\leq F(G)\leq n-2$ and gave a characterization of all graphs $G$ where $F(G)=0,1,n-1$ and $n-2$. They showed that the only graphs where $F(G)=0$ are either $K_1$ or $K_2$ and the only graphs where $F(G)=1$ are $2K_1$, $P_3$, $K_3$, or $P_4$. It was also shown that $F(G)=n-1$ if and only if $G$ contains an isolated vertex.  In addition, $F(G)=n-2$ if and only if $G$ contains a module of order $2$, where a module is defined as follows. A set $X$ of vertices in $V(G)$ is a module if all vertices in $X$ have
the same set of neighbors among vertices not in $X$. They gave examples of some graphs where $F(G)=2$ and posed the problem of characterizing all graphs with this property. We provide a complete solution to this problem in Section 2.

\section{Graphs with a Failed Zero-forcing number of 2}

In this paper we provide a characterization of all graphs with a failed zero-forcing number of 2. We first examined all graphs with between 3 and 6 vertices and identified all graphs with $F(G)=2$. These are shown in Figure 1. Later we show that no other graphs exist. We will show that if $G$ has at least 7 vertices then $F(G) \geq 3$.

  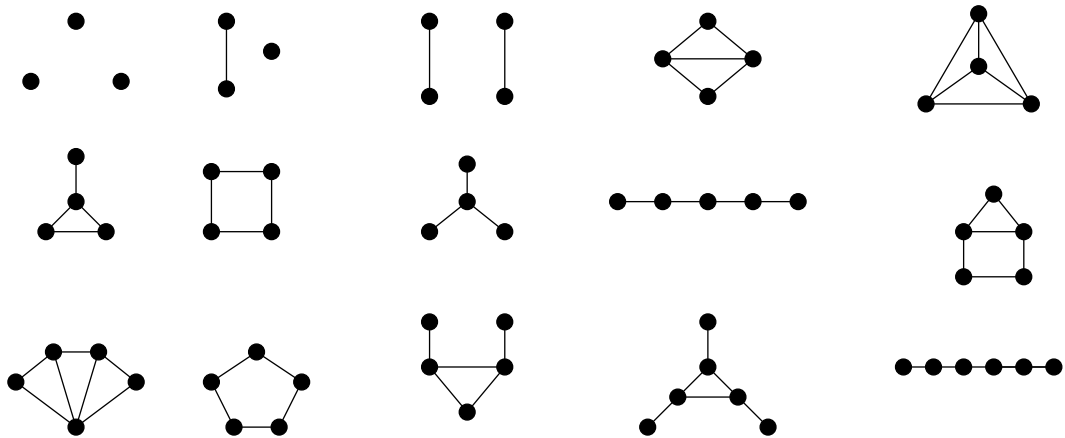
\begin{figure}[h!]
\begin{center}
\begin{tikzpicture}[node distance = 1cm, line width = 0.5pt]

\coordinate (1) at (0,0);
\coordinate (2) at (0,0.5);
\coordinate (3) at (0.5,-0.4);
\coordinate (4) at (-0.5,-0.4);
\coordinate (5) at (-0.5,1.4);
\coordinate (6) at (-0.5,2.4);
\coordinate (7) at (0.5,1.4);
\coordinate (8) at (0.5,2.4);
\coordinate (9) at (0,-2.8);
\coordinate (10) at (0.5,-2.2);
\coordinate (11) at (-0.5,-2.2);
\coordinate (12) at (0.5,-1.6);
\coordinate (13) at (-0.5,-1.6);
\coordinate (14) at (2,0);
\coordinate (15) at (2.6,0);
\coordinate (16) at (3.2,0);
\coordinate (17) at (3.8,0);
\coordinate (18) at (4.4,0);
\coordinate (19) at (3.2,1.4);
\coordinate (20) at (3.2,2.4);
\coordinate (21) at (2.6,1.9);
\coordinate (22) at (3.8,1.9);
\coordinate (23) at (6.8,1.8);
\coordinate (24) at (6.8,2.5);
\coordinate (25) at (6.1,1.3);
\coordinate (26) at (7.5,1.3);
\coordinate (27) at (6.6,-1);
\coordinate (28) at (7.4,-1);
\coordinate (29) at (6.6,-0.4);
\coordinate (30) at (7.4,-0.4);
\coordinate (31) at (7,0.1);
\coordinate (32) at (7,-2.2);
\coordinate (33) at (7.4,-2.2);
\coordinate (34) at (7.8,-2.2);
\coordinate (35) at (6.6,-2.2);
\coordinate (36) at (6.2,-2.2);
\coordinate (37) at (5.8,-2.2);
\coordinate (38) at (3.2,-2.2);
\coordinate (39) at (3.2,-1.6);
\coordinate (40) at (3.6,-2.6); 
\coordinate (41) at (2.8,-2.6); 
\coordinate (42) at (2.4,-3.0); 
\coordinate (43) at (4,-3.0); 
\coordinate (44) at (-3.4,0.4); 
\coordinate (45) at (-2.6,0.4); 
\coordinate (46) at (-3.4,-0.4); 
\coordinate (47) at (-2.6,-0.4); 
\coordinate (48) at (-2.6,2); 
\coordinate (49) at (-3.2,2.4); 
\coordinate (50) at (-3.2,1.6); 
\coordinate (51) at (-5.2,2.4);
\coordinate (51) at (-4.6,1.6);
\coordinate (51) at (-5.8,1.6);
\coordinate (52) at (-5.2,0);
\coordinate (53) at (-5.2,0.6);
\coordinate (54) at (-5.6,-0.4);
\coordinate (55) at (-4.8,-0.4);
\coordinate (56) at (-5.2,-3);
\coordinate (57) at (-6,-2.4);
\coordinate (58) at (-4.4,-2.4);
\coordinate (59) at (-4.9,-2);
\coordinate (60) at (-5.5,-2);
\coordinate (61) at (-2.8,-2);
\coordinate (62) at (-3.4,-2.4);
\coordinate (63) at (-2.2,-2.4);
\coordinate (64) at (-3.1,-3.0);
\coordinate (65) at (-2.5,-3.0);

\draw (1)--(2);
\draw (1)--(3);
\draw (1)--(4);
\draw (5)--(6);
\draw (7)--(8);
\draw (9)--(10);
\draw (9)--(11);
\draw (10)--(11);
\draw (10)--(12);
\draw (11)--(13);
\draw (14)--(15);
\draw (15)--(16);
\draw (16)--(17);
\draw (17)--(18);
\draw (19)--(21);
\draw (21)--(22);
\draw (19)--(22);
\draw (20)--(21);
\draw (20)--(22);
\draw (23)--(24);
\draw (23)--(25);
\draw (23)--(26);
\draw (24)--(25);
\draw (24)--(26);
\draw (25)--(26);
\draw (27)--(28);
\draw (27)--(29);
\draw (28)--(30);
\draw (29)--(30);
\draw (29)--(31);
\draw (30)--(31);
\draw (32)--(33);
\draw (33)--(34);
\draw (34)--(35);
\draw (35)--(36);
\draw (36)--(37);
\draw (38)--(39);
\draw (38)--(40);
\draw (38)--(41);
\draw (40)--(41);
\draw (41)--(42);
\draw (40)--(43);
\draw (44)--(45);
\draw (46)--(47);
\draw (44)--(46);
\draw (45)--(47);
\draw (49)--(50);
\draw (52)--(53);
\draw (52)--(54);
\draw (54)--(55);
\draw (52)--(55);
\draw (56)--(57);
\draw (56)--(58);
\draw (56)--(59);
\draw (56)--(60);
\draw (57)--(60);
\draw (59)--(60);
\draw (58)--(59);
\draw (61)--(62);
\draw (61)--(63);
\draw (62)--(64);
\draw (63)--(65);
\draw (64)--(65);

\filldraw [black] 
(0,0) circle (3pt)
(0,0.5) circle (3pt)
(0.5,-0.4) circle (3pt)
(-0.5,-0.4) circle (3pt)
(-0.5,1.4) circle (3pt)
(-0.5,2.4) circle (3pt)
(0.5,1.4) circle (3pt)
(0.5,2.4) circle (3pt)
(0,-2.8) circle (3pt)
(-0.5,-2.2) circle (3pt)
(0.5,-2.2) circle (3pt)
(-0.5,-1.6) circle (3pt)
(0.5,-1.6) circle (3pt)
(2,0) circle (3pt)
(2.6,0) circle (3pt)
(3.2,0) circle (3pt)
(3.8,0) circle (3pt)
(4.4,0) circle (3pt)
(3.2,1.4) circle (3pt)
(3.2,2.4) circle (3pt)
(3.8,1.9) circle (3pt)
(2.6,1.9) circle (3pt)
(6.8,1.8) circle (3pt)
(6.8,2.5) circle (3pt)
(6.1,1.3) circle (3pt)
(7.5,1.3) circle (3pt)
(6.6,-1) circle (3pt)
(7.4,-1) circle (3pt)
(6.6,-0.4) circle (3pt)
(7.4,-0.4) circle (3pt)
(7,0.1) circle (3pt)
(7,-2.2) circle (3pt)
(7.4,-2.2) circle (3pt)
(7.8,-2.2) circle (3pt)
(6.6,-2.2) circle (3pt)
(6.2,-2.2) circle (3pt)
(5.8,-2.2) circle (3pt)
(3.2,-2.2) circle (3pt)
(3.2,-1.6) circle (3pt)
(3.6,-2.6) circle (3pt)
(2.8,-2.6) circle (3pt)
(2.4,-3.0) circle (3pt)
(4,-3.0) circle (3pt)
(-3.4,0.4) circle (3pt)
(-3.4,-0.4) circle (3pt)
(-2.6,0.4) circle (3pt)
(-2.6,-0.4) circle (3pt)
(-2.6,2) circle (3pt)
(-3.2,2.4) circle (3pt)
(-3.2,1.5) circle (3pt)
(-5.2,2.4) circle (3pt)
(-4.6,1.6) circle (3pt)
(-5.8,1.6) circle (3pt)
(-5.2,0) circle (3pt)
(-5.2,0.6) circle (3pt)
(-4.8,-0.4) circle (3pt)
(-5.6,-0.4) circle (3pt)
(-5.2,-3) circle (3pt)
(-6,-2.4) circle (3pt)
(-4.4,-2.4) circle (3pt)
(-5.5,-2) circle (3pt)
(-4.9,-2) circle (3pt)
(-2.8,-2) circle (3pt)
(-2.2,-2.4) circle (3pt)
(-3.4,-2.4) circle (3pt)
(-2.5,-3.0) circle (3pt)
(-3.1,-3.0) circle (3pt);
\end{tikzpicture}
\end{center}
\caption{The 15 graphs with $F(G)=2$}
\end{figure}

It is not difficult to verify that each of these graphs have a failed zero-forcing number of $2$. For the sake of completeness, we include the details. The coloring of the vertices gives a lower bound for the failed zero-forcing number of $2$. For the disconnected three graphs the upper bound for $F(G)=2$ follows since coloring any three vertices blue forces all of the vertices to be colored blue.  For the four connected graphs with four vertices the upper bound follows since $F(G)\leq n-2$. For the two paths, the coloring of any three vertices blue will result in either a vertex of degree 1 being colored blue or two adjacent vertices colored blue, both of which will color all of the vertices in the graph. For the graph of $C_{5}$ coloring any three vertices blue will mean that two adjacent vertices were colored blue, which will force the remaining vertices in the graph. For the graph of $C_{5}$ and a chord note that coloring any three of the vertices in the graph blue will either include two vertices of the triangle or three vertices in the $C_{4}$ both of which will force all vertices in the graph to be blue. For the wheel graph with five vertices, coloring three vertices blue will either include the vertex of degree four and two other vertices, or three vertices of degree two or three, both of which will force all vertices in the graph to be blue. For the graph consisting of a triangle and two pendant edges note that coloring any three vertices blue will either include the three vertices of degree greater than 1, one vertex of degree 1 and two vertices of degree greater than 1, or two vertices of degree 1, all of which will force all vertices to be colored blue. Finally for the triangle with three pendant edges, coloring any three vertices blue will either include all three vertices of degree greater than 1, or at least one vertex of degree 1 and two other vertices that are either degree 1 or greater than 1, all of which will force all vertices to be blue.

We examined all connected graphs on six vertices \cite{Cvetković} and found that all but two graphs have a failed zero forcing number of at least 3. The two graphs $P_6$ and the corona of $K_3$, (the last two graphs in Figure 1) have a failed zero forcing number of $2$. Adding an additional vertex $v$ and any number of edges between $v$ and other vertices of the graph results in a graph with a failed zero forcing number of 3 (This is verified later in Section 2.1). The above suggests that any graph with 7 or more vertices has a failed zero forcing number of at least 3, which we will prove.

The following theorem guarantees that our list of 15 graphs with $F(G)=2$ is complete.

\begin{theorem}
Let $G$ be a graph with at least $7$ vertices. Then $F(G^{\prime })\geq 3$.
\end{theorem} 

 \bigskip 
 It is tempting to think that if you have a graph $G$ where $F(G)=k$, then for any supergraph $H$ of $G$ we have $F(H)\geq k$. However it can be the case that $F(H)<F(G)$ and in fact we can construct graphs where $F(G)-F(H)$ is arbitrarily large. In the figure below we give an example of a graph $G$ where $F(G)=5$ and a supergraph $H$ of $G$ where $F(H)=4$.
 
 \begin{figure}[h!]
\begin{center}
\begin{tikzpicture}[node distance = 1cm, line width = 0.5pt]
\coordinate (1) at (0,0);
\coordinate (2) at (0.75,0);
\coordinate (3) at (1.5,0);
\coordinate (4) at (2.25,0);
\coordinate (5) at (3,0);
\coordinate (6) at (3.75,0);
\coordinate (7) at (3.75,-0.6);

\coordinate (9) at (5.25,0);
\coordinate (10) at (6,0);
\coordinate (11) at (6.75,0);
\coordinate (12) at (7.5,0);
\coordinate (13) at (8.25,0);
\coordinate (14) at (9,0);
\coordinate (15) at (9,-0.6);

\draw \foreach \x [remember=\x as \lastx (initially 1)] in {2,3,4,5,6}{(\lastx) -- (\x)};
\draw (5)--(7);
\draw (9)--(10);
\draw (10)--(11);
\draw (11)--(12);
\draw (12)--(13);
\draw (13)--(14);
\draw (13)--(15);
\draw[solid] (9) to [out=45, in=135] (14);

\foreach \point in {6,7,9,10,12,14,15} \fill (\point) circle (4pt);
\foreach \point in {1,2,3,4,5,9,11,13} \fill (\point) 
circle (4pt);

\filldraw [blue] 
(0,0) circle (3pt)
(0.75,0) circle (3pt)
(1.5,0) circle (3pt)
(2.25,0) circle (3pt)
(3,0) circle (3pt)
(5.25,0) circle (3pt)
(6.75,0) circle (3pt)
(8.25,0) circle (3pt)
(9,-0.6) circle (3pt);

\filldraw [white] 
(3.75,0) circle (3pt)
(3.75,-0.6) circle (3pt)
(6,0) circle (3pt)
(7.5,0) circle (3pt)
(9,0) circle (3pt);


\end{tikzpicture}
\end{center}
\caption{Adding a single edge to $G$ to create a supergraph $H$}
\end{figure}
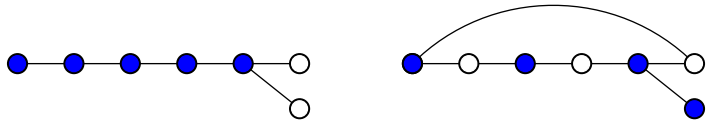

\begin{theorem}
There exist graphs $G$ with a supergraph $H$ where $F(G)-F(H)$ is arbitrarily large.
\end{theorem} 
\begin{proof}

Let $G$ be a graph with vertices $v_{1},v_{2},...,v_{n}$ with edges $\{v_{i}v_{i+1}|$ $i=1,...,n-2\}$ and $\{v_{n-2}v_{n}\}$. Here $F(G)=n-2$.

Now add the edge $v_{1}v_{n}$ to create a supergraph $H$. Since $H$ is a cycle with a pendant vertex we can obtain a maximum-sized failed zero-forcing set by selecting the pendant vertex and $\left\lfloor \frac{n}{2}\right\rfloor $ vertices along the cycle (including the neighbor of the pendant vertex). Hence $F(H)=\left\lfloor \frac{n}{2}\right\rfloor +1$. Since $F(G)-F(H)=n-2-\left( \left\lfloor \frac{n}{2}\right\rfloor +1\right) \geq \frac{n}{2}$ this difference can be made arbitrarily large, but choosing a graph with a sufficiently large number of vertices.
\end{proof}

We next present a series of lemmas that will be used to prove Theorem 2.1.
\medskip
\newline
A characterization of disconnected graphs with a failed zero-forcing number $k$ was given by Fetcie, Jacob, and Saavedra \cite{Fetcie}.

\begin{lemma}
Let $G$ be a disconnected graph with $k\geq 2$ maximal connected components. Let $G_{1},G_{2},...,G_{k}$ be the $k$ maximal connected components of $G$. Then $F(G)=\max \left\{ F(G_{k})+\sum\limits_{l\neq k}\left\vert V(G_{l})\right\vert \right\}$.
\end{lemma}

This theorem can be applied to determine all disconnected graphs with a particular failed zero-forcing number. We show the result for the case where $F(G)=2$.

\begin{lemma}
For a disconnected graph $G$, $F(G)=2$ if and only if $G$ equals $3K_{1}$, $K_{2}+K_{1}$, or $K_{2}+K_{2}$.
\end{lemma}

\begin{proof}
If $G$ consists of three components the only possibility is for $G$ to equal $3K_{1}$, since if any component contained $k>1$ vertices we would have $F(G)\geq k+1=3$. If $G$ has two components one component must contain exactly two vertices, and the other component must have a failed zero-forcing number of 1, which are either $K_1$ or $K_2$.  
\end{proof}

Now that we have identified all disconnected graphs with $F(G)=2$, we next consider the case of connected graphs. We note that if a graph with 8 or more vertices contains a cherry or a pendant triangle then $F(G)\geq 3$. 

The approach will be as follows. 
\bigskip
\bigskip
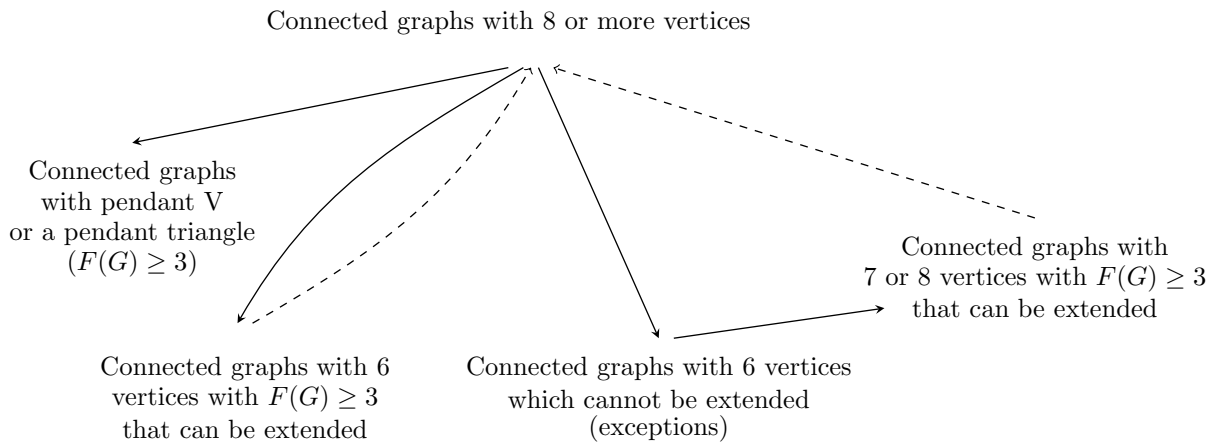
\begin{figure}
\begin{center}
\begin{tikzpicture}[node distance = 1cm, line width = 0.5pt]


\draw [-stealth](0,2) -- (-5,1);
\draw[-stealth] (0.2,2) to [out=210, in=60] (-3.6,-1.4);
\draw[dashed,->] (-3.4,-1.4) to [out=30, in=240] (0.3,2);
\draw [-stealth](0.4,2) -- (2,-1.6);
\draw [-stealth](2.2,-1.6) -- (5,-1.2);
\draw[dashed,->](7,0) -- (0.6,2);



\node (A) at (-5,0.6) {Connected graphs};
\node (B) at (-5,0.2) {with pendant V};
\node (C) at (-5,-0.2) {or a pendant triangle};
\node (C2) at (-5,-0.6) {($F(G)\geq 3$)};

\node (D) at (0,2.6) {Connected graphs with 8 or more vertices};

\node (E) at (-3.5,-2) {Connected graphs with 6};
\node (F) at (-3.5,-2.4) {vertices with $F(G)\geq 3$};
\node (G) at (-3.5,-2.8) {that can be extended};

\node (H) at (2,-2) {Connected graphs with 6 vertices};
\node (I) at (2,-2.4) {which cannot be extended};
\node (J) at (2,-2.8) {(exceptions)};

\node (H) at (7,-0.4) {Connected graphs with};
\node (I) at (7,-0.8) { 7 or 8 vertices with $F(G)\geq 3$};
\node (J) at (7,-1.2) {that can be extended};

\end{tikzpicture}
\caption{Dotted arrows indicate extensions by Lemma 2.4}

\end{center}
\end{figure}

A pendant V is a subgraph containing a cut-vertex $v$ of degree at least $3$ whose removal results in at least two paths in different components. 

We start with a connected graph with 8 or more vertices that does not contain a pendant V or a pendant triangle and reduce it to a graph with 6 vertices by removing vertices and incident edges. We note we can maintain a connected graph by successively removing a leaf from a spanning tree of the graph. We will show in Lemmas 2.6 and 2.7 that the reduced graphs will not contain a pendant V or a pendant triangle. If the resulting graph on six vertices does not have any true blue vertices, then we apply Lemma 2.3 add back vertices and edges to obtain the original graph which will have a failed zero-forcing set of size at least 3. If the resulting graph on six vertices has at least one true blue vertex then we consider all possible extensions to supergraphs with 7 vertices. We show that all of these graphs have $F(G)\geq 3$. If these graphs on 7 vertices do not have any true blue vertices then we can apply Lemma 2.3 and extend the graph to the original graph which will a failed zero-forcing set of size at least 3. If the extensions to 7 vertices have at least one true blue vertex then we consider all possible extensions to 8 vertices. We were able to show that all of these graphs do not have any true blue vertices so they can be extended back to the original graph showing that $F(G)\geq 3$.

\begin{lemma}
    Let $G$ be a graph with a failed zero-forcing set $S$. If every vertex in $S$ is adjacent to two vertices outside of $S$, then $S$ is a failed-zero forcing set for any connected supergraph $H$ of $G$, and thus $F(H)\geq F(G)$.
\end{lemma}
\begin{proof}
When adding vertices and edges to $G$ to obtain the graph $H$ we can color all of the vertices in $V(H)-V(G)$ white. Since each member of $S$ is adjacent to two vertices not in $S$, $S$ will not force any vertices in $H$.
\end{proof}

In our next lemma we show that a if a graph $G$ can be extended to contain an additional vertex $v$ and if the extended graph has a failed zero-forcing set $S$ where $v$
is not in $S$ then any number of edges can be added between $v$ and vertices of $G$ and $S$ will still be a failed zero-forcing set of any of these new graphs. 
\begin{lemma}
 Let $G$ be a connected graph with a set of $S$ of blue vertices. If each vertex in $S$ is adjacent to two vertices not in $S$, then for any supergraph $H$ of $G$, we have that $F(H)\geq F(G)$.
 \begin{proof}
 We can add any number of vertices and edges $G$ to create a supergraph $H$. Then color all of the vertices in $V(G)-V(H)$ white. Then since white vertices cannot force any other vertices and the blue vertices are still adjacent to two white vertices, no vertices in the graph will be forced. Hence $F(H)\geq F(G)$.
 \end{proof}
\end{lemma}

\begin{lemma}
If a connected graph $G$ has a pendant $K_3$. Then $F(G)\geq n-2$.
\end{lemma}
	
\begin{proof}
All of the vertices except the two vertices of degree $2$ in the $K_3$ form a failed zero forcing set of $|V(G)|-2$ which is the largest possible for a connected graph.
\end{proof}

\medskip

The following lemma shows that if $G$ is a graph with at least $5$ vertices then $F(G)\geq 3$.

\begin{lemma}
 Let $G$ be a graph with $5$ or more vertices and contains a cut-vertex $v$ such that $G-v$ has at least three components, at least two of which are paths. Then $F(G)\geq 3$.
\end{lemma}

\begin{proof}

We consider two cases.
	
Case (i) $G$ has a cut-vertex $v$ such that $G-v$ consists of three disjoint paths $P_{i}$, $P_{j}$, and $P_{k}$. Then either $i$, $j$, or $k$ must be at least $2$. Without loss of generality suppose $i\geq 3$. Then $v$ along with the vertices of $P_{i}$ form a failed zero forcing set of at least $3$. 
	
Case (ii) $G$ is a graph with $5$ or more vertices and has a cut-vertex $v$ such that $G-v$ consists of a subgraph $H$ with at least two vertices and two disjoint paths. Then $v$ along with the vertices of $H$ form a failed zero forcing set of at least $3$. 
\end{proof}

In our next two lemmas we prove that if a graph $G$ is a graph does contains a pendant V or a pendant $K_3$ then it can be reduced to a subgraph $H$ by removing vertices and edges, so that $H$ does not contain a pendant V or a pendant $K_3$. If all reductions from a graph $G$ lead to a graph $H$ with one of these two pendant subgraphs, then $G$ contained one of these two pendant subgraphs.

\begin{lemma}[pendant V]
 Let $G$ be a connected graph with at least five vertices that does not contain a pendant V. Then $\left\vert V(G)\right\vert -5$ vertices of $G$ can be successively removed (along with incident edges) to create a graph $G^{\prime }$ that does not have a pendant V.
\end{lemma}
\begin{proof}
We will show that if a graph $G$ does not have a pendant V, then there is some vertex we can remove from $G$ so that the resulting graph does not have a pendant V. Let $G$ be a graph that is not a cycle and does not have a pendant $V$ subgraph. We want to show there is always some vertex $v$ whose removal does not leave a pendant $V$ subgraph. We can assume that $\delta (G)\geq 2$ since we could remove a vertex of degree $1$ and not create a pendant $V$ subgraph. Since $\delta (G)\geq 2$, $G$ must contain a cycle $C$ that has no pendant paths. Since $G\neq C$ the cycle $C$ must contain a vertex $w$ with degree at least $3$. Removing a vertex on $C$ that has degree 2 is adjacent to $w$ will not leave a pendant $V$.
\end{proof}

\begin{lemma}[pendant triangle]
 Let $G$ be a connected graph with at least five vertices that does not contain a pendant $K_{3}$. Then $\left\vert V(G)\right\vert -5$ vertices of $G$ can be successively removed (along with incident edges) to create a graph $G^{\prime }$ that does not have a pendant $K_{3}$.
\end{lemma}
\begin{proof}

We will show that if $G$ does not contain a pendant triangle, there is always some vertex we can remove so that the resulting graph does not contain a pendant triangle. Let $G$ be a graph without a pendant triangle, but $G-v$ contains a pendant triangle with vertices $a$,$b$,$c$, and $d$ where $a$ and $b$ are adjacent, and $b$,$c$, and $d$ are all mutually adjacent, and $c$ and $d$ are not adjacent to $a$. Since $G$ has more than $5$ vertices we can assume that either $d(a)>2$ or  $d(v)>2$. We will consider various cases and show in each case that a different vertex can be removed from $G$ that will not result in a pendant triangle.
	
Case (i):\ $v$ is adjacent to $a$,$b$,$c$, and $d$. Removing $d$ leaves an induced $K_{4}-e$ on vertices $a$,$v$,$b$, and $c$.
	
Case (ii):\ $v$ is adjacent to $a$,$b$, and $c$ only. Then removing $b$ results in a path on vertices $a$,$v$,$c$, and $d$.
	
Case (iii):\ $v$ is adjacent to $c$ and $d$ only. Then removing $c$ results in a path on vertices $v$,$d$,$b$, and $a$.
	
Case (iv):\ $v$ is adjacent to $c$ only. Then removing $d$ results in a path on vertices $v$,$c$,$b$, and $a$.
	
Case (v):\ $v$ is adjacent to $d$ only. Then removing $d$ results in a path on vertices $v$,$c$,$b$, and $a$.
\end{proof}

\medskip

\subsection{Graphs with six vertices}

We will consider the 112 connected graphs with six vertices. We will refer to the table and numbering found in \cite{Cvetković}. For convenience, we explicitly show each graph in each of the individual cases.

Out of $112$ different graphs, there were 16 cases where there does not exist three blue vertices with at most one true blue vertex. We will show that show that all possible extensions of these exception cases have three blue vertices with none being a true blue vertex. This will not only show that they have a failed zero-forcing number of at least 3, but they can be extended indefinitely by adding vertices and edges and the resulting supergraph will also have a failed zero-forcing number of at least 3.

These graphs are the exceptions: 54,58,59,60,76,77,80,85,86,87,96,98,102,103,105,112, each of which we will present individually.

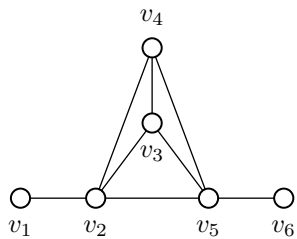
\begin{figure}[h!]
\begin{center}
\begin{tikzpicture}[node distance = 1cm, line width = 0.5pt]
\coordinate (1) at (0,0);
\coordinate (2) at (1,0);
\coordinate (3) at (2.5,0);
\coordinate (4) at (1.75,1);
\coordinate (5) at (1.75,2);
\coordinate (6) at (3.5,0);


\draw (1)--(2);
\draw (2)--(3);
\draw (2)--(4);
\draw (2)--(5);
\draw (3)--(6);
\draw (3)--(4);
\draw (3)--(5);
\draw (4)--(5);

\foreach \point in {1,2,3,4,5,6} \fill (\point) circle (4pt);

\filldraw [white] 
(1.75,1) circle (3pt)
(3.5,0) circle (3pt)
(1.75,2) circle (3pt)
(0,0) circle (3pt)
(1,0) circle (3pt)
(2.5,0) circle (3pt);

\node (A) at (0,-0.4) {$v_1$};
\node (B) at (1,-0.4) {$v_2$};
\node (C) at (2.5,-0.4) {$v_5$};
\node (D) at (1.75,2.4) {$v_4$};
\node (E) at (3.5,-0.4) {$v_6$};
\node (F) at (1.75,0.6) {$v_3$};

\end{tikzpicture}
\end{center}
\caption{Graph 54}
\end{figure}

\textbf{Graph 54:} 
We begin by extending graph 54 by one vertex, $v$ and connecting it to each vertex in G. Notice there will be no case where $v$ is adjacent to $v_2$ or $v_5$. This is because connecting a vertex to either one will create a cherry structure and that case is covered by our cherry lemma. Our goal with extending is to find no true blue vertex within the graph. It is also important to notice that in each case, we are able to find a failed zero-forcing number of three.

If $v$ is adjacent to $v_1$, then $v_1,v_3,$ and $v_4$ can be filled.

If $v$ is adjacent to $v_6$, then $v_3,v_4,$ and $v_6$ can be filled.

If $v$ is adjacent to $v_4$, then $v_2,v_5,$ and $v_6$ can be filled. Notice that $v_6$ is true blue, so this case must be extended further. Suppose we add a new vertex called $u$.

If $u$ is adjacent to $v_3$, then $v_3,v_4,$ and $v_5$ can be filled.

If $u$ is adjacent to $v_6$, then $v_3,v_4,$ and $v_6$ can be filled.

If $u$ is adjacent to $v$, then $v,v_2,$ and $v_5$ can be filled. Since all sub-cases for this case have no true blue vertex, then this case is resolved. But there remains one more case.

If $v$ is adjacent to $v_3$, then $v_2,v_5,$ and $v_6$ can be filled. Notice that $v_6$ is true blue, so this case must be extended further. Suppose we add a new vertex called $u$.

If $u$ is adjacent to $v$, then $v,v_2,$ and $v_5$ can be filled.

If $u$ is adjacent to $v_6$, then $v_3,v_4,$ and $v_6$ can be filled.

If $u$ is adjacent to $v_4$, then $v_3,v_4,$ and $v_5$ can be filled. Notice in each sub-case there are no true blue vertices. Thus this case is resolved and can be extended. Since all cases can be extended indefinitely, this implies that graph 54 can be extended indefinitely without being forced with $F(G) \geq 3$.

\begin{figure}[h!]
    \begin{center}
        \begin{tikzpicture}
        \coordinate (1) at (1,1);
        \coordinate (2) at (1,0.5);
        \coordinate (3) at (2,0); 
        \coordinate (4) at (1.5,-1);
        \coordinate (5) at (0.5,-1);
        \coordinate (6) at (0,0);
        \draw \foreach \x [remember=\x as \lastx (initially 1)] in {1,2,3,4,5,6}{(\lastx) -- (\x)};
        \draw (2) -- (6);
        \draw (2) -- (4);
        \draw (2) -- (5);
        \foreach \point in {3,4,6} \fill (\point) circle (4pt);
        \foreach \point in {1,2,5} \fill (\point) circle (4pt);
   
        \filldraw [white]
        (0,0) circle (3 pt)
        (2,0) circle (3 pt) 
        (1.5,-1) circle (3 pt)
        (1,1) circle (3 pt) 
        (1,0.5) circle (3 pt) 
        (0.5, -1) circle (3 pt);
        
        \node (A) at (0,0.4) {$v_6$};
        \node (B) at (1,1.4) {$v_1$};
        \node (C) at (1.4,0.6) {$v_2$};
        \node (D) at (2,0.4) {$v_3$};
        \node (E) at (1.9,-1) {$v_4$};
        \node (F) at (0.1,-1) {$v_5$};
        \end{tikzpicture}
        \caption{Graph 58}
    \end{center}
\end{figure}
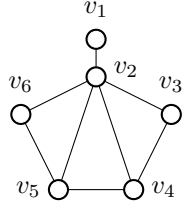

\textbf{Graph 58: } Let $v_1$ be the vertex of degree $1$ and let $v_2$ be the neighbor of $v_1$; $v_4$ is the vertex of degree $3$ that is adjacent to $v_2$, $v_3$, and $v_5$. Let $v_5$ be the vertex of degree $3$ that is adjacent to $v_6$; vertices $v_6$ and $v_3$ both have degree 2.  With graph 58, although we have one True blue vertex, we can extend the graph along vertices $v_1, v_3$ and $v_6$ and be left without any true blue vertices thus being able to extend indefinitely with $F(G) \geq 3$. 
Extending along vertices $v_4$ and $v_5$ will give us one true blue vertex.
We can then add an additional vertex or an edge and with reshuffling the filled in vertices, that will give us $F(G) \geq 3$ without a true blue vertex. For the sake of completeness we include the details.  We consider adding a new vertex $v$ with possible neighbors. We note that $v$ cannot have degree $1$ and be adjacent to $v_2$ since this would create a cherry. In each case of the remaining cases we can find a coloring with $3$ blue vertices, none of which are true blue vertices, and all of the remaining vertices including $v$ will be colored white. If $v$ is adjacent to $v_1$, then we could color $v_1$, $v_3$, and $v_5$ blue. If $v$ is adjacent to $v_3$, then we could color $v_2$, $v_3$, and $v_5$ blue. If $v$ is adjacent to $v_6$, then we could color $v_1$, $v_5$, and $v_6$ blue. The only problem cases is where $v$ is adjacent to $v_4$ or $v_5$. Here we can find a failed zero forcing set $S$ of size $3$ by selecting $v_1$, $v_2$, and $v_4$, however we may not be able to extend $S$ to larger graphs.  Without loss of generality assume that $v$ is adjacent to $v_4$. We consider adding a new vertex $u$ with possible neighbors. We note that $u$ cannot have degree $1$ and be adjacent to $v_2$ or $v_4$ since this would create a cherry. If $u$ is adjacent to $v$ then we could color $v$, $v_3$, and $v_5$ blue. If $u$ is adjacent to $v_1$ then we could color $v_1$, $v_3$, and $v_5$ blue. If $u$ is adjacent to $v_3$ then we could color $v_2,$ $v_4$, and $v_5$ blue. If $u$ is adjacent to $v_5$ then we could color $v_2,$ $v_4$, and $v_5$ blue. If $u$ is adjacent to $v_6$ then we could color $v_3,$ $v_5$, and $v_6$ blue.

\begin{figure}[h!]
\begin{center}
\begin{tikzpicture}[node distance = 1cm, line width = 0.5pt]
\coordinate (1) at (0.5,0);
\coordinate (2) at (1,-1);
\coordinate (3) at (2.0,-1);
\coordinate (4) at (2.5,0);
\coordinate (5) at (1.5,1);
\coordinate (6) at (3.5,-1);

\draw (6)--(3);
\draw (5)--(1);
\draw (5)--(4);
\draw (5)--(3);
\draw (5)--(2);
\draw (1)--(2);
\draw (2)--(3);
\draw (3)--(4);

\foreach \point in {1,2,3,4,5,6} \fill (\point) circle (4pt);

\filldraw [white] 
(0.5,0) circle (3pt)
(1.5,1) circle (3pt)
(2.5,0) circle (3pt)
(1,-1) circle (3pt)
(2.0,-1) circle (3pt)
(3.5,-1) circle (3pt);

\node (A) at (0,0) {$v_6$};
\node (B) at (0.6,-1) {$v_5$};
\node (C) at (2.4,-0.7) {$v_4$};
\node (D) at (2.5,0.4) {$v_3$};
\node (E) at (1.1,1) {$v_2$};
\node (F) at (3.5,-0.6) {$v_1$};


\end{tikzpicture}
\end{center}
\caption{Graph 59}
\end{figure}
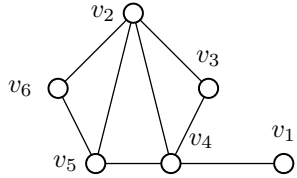

\textbf{Graph 59: } Let $v_1$ be the vertex of degree $1$, $v_4$ be the vertex of degree $4$ that is adjacent to $v_1$, $v_3$ be the vertex of degree $2$ that is adjacent to $v_4$, $v_2$ be the other vertex of degree $4$, $v_5$ be the vertex of degree $3$ that is adjacent to $v_2$ and $v_4$, and $v_6$ be the vertex of degree $2$ that is adjacent to $v_2$ and $v_5$. Graph 59 follows with the case of graph 58. We can extend along any vertex without any True Blue vertices with the only exception being when a new vertex is adjacent to vertex $v_3$. Similarly to graph 58 however, upon any additional extension, there are no issues arising when all the vertices are reshuffled.

For the sake of completeness we include the details.  We consider adding a new vertex $v$ and possible neighbors of $v$. The vertex $v$ cannot have degree $1$ and be adjacent to $v_4$ since this would create a cherry. In each case of the remaining cases we can find a coloring with $3$ blue vertices, none of which are true blue vertices, and all of the remaining vertices including $v$ will be colored white. If $v$ is adjacent to $v_1$, then $v_1,$ $v_3$, and $v_5$ can be colored blue. If $v$ is adjacent to $v_3$, then $v_3,$ $v_4$, and $v_6$ can be colored blue. If $v$ is adjacent to $v_5$ then $v_2,$ $v_4,$ and $v_5$ can be colored blue. If $v$ is adjacent to $v_6$, then $v_2$, $v_4,$ and $v_6$ can be colored blue. The only problem case is if $v$ is adjacent to $v_2$. Here we can find a failed zero-forcing set $S$ of size $3$ with vertices $v_1$, $v_4$, and $v_6$, but we many not be able to extend this graph. Here we consider adding another vertex $u$. The vertex $v$ cannot have degree $1$ and be adjacent to either  $v_2$ or $v_4$ since this would create a cherry. If $u$ is adjacent to $v$, then we color $v$,$v_3$, and $v_5$ blue. If $u$ is adjacent to $v_1$, then we color $v_1$, $v_3$, and $v_5$ blue. If $u$ is adjacent to $v_3$, then we color $v_2$, $v_3$, and $v_5$ blue. If $u$ is adjacent to $v_5$, then we color $v_2$, $v_4$, and $v_5$ blue. If $u$ is adjacent to $v_6$, then we color $v_2$,$v_4$, and $v_6$ blue.

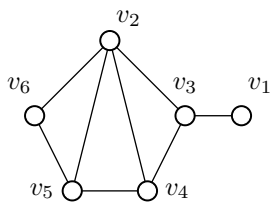
\begin{figure}[h!]
\begin{center}
\begin{tikzpicture}[node distance = 1cm, line width = 0.5pt]
\coordinate (1) at (0.5,0);
\coordinate (2) at (1,-1);
\coordinate (3) at (2.0,-1);
\coordinate (4) at (2.5,0);
\coordinate (5) at (1.5,1);
\coordinate (6) at (3.25,0);

\draw (6)--(4);
\draw (5)--(1);
\draw (5)--(4);
\draw (5)--(3);
\draw (5)--(2);
\draw (1)--(2);
\draw (2)--(3);
\draw (3)--(4);

\foreach \point in {1,2,3,4,5,6} \fill (\point) circle (4pt);

\filldraw [white] 
(1,-1) circle (3pt)
(2.0,-1) circle (3pt)
(1.5,1) circle (3pt)
(0.5,0) circle (3pt)
(2.5,0) circle (3pt)
(3.25,0) circle (3pt);

\node (A) at (0.3,0.4) {$v_6$};
\node (B) at (0.6,-1) {$v_5$};
\node (C) at (2.4,-1) {$v_4$};
\node (D) at (2.5,0.4) {$v_3$};
\node (E) at (3.5,0.4) {$v_1$};
\node (F) at (1.75,1.3) {$v_2$};

\end{tikzpicture}
\end{center}
\caption{Graph 60}
\end{figure}

\textbf{Graph 60: } Graph 60 has at least one true blue vertex but all extensions along graph 60 result in $F(G) \geq 3$ and no true blue vertices. Let $v_1$ be the vertex of degree $1$, $v_3$ be the neighbor of $v_1$, $v_2$ be the vertex of degree $4$, $v_4$ be the vertex of degree $3$ that is adjacent to $v_2$ and $v_3$, $v_5$ be the vertex of degree $3$ adjacent to $v_4$, and let $v_6$ be the vertex of degree $2$. We consider adding another vertex $v$ and possible neighbors of $v$. First if $v$ has degree $1$ then it cannot be adjacent to $v_3$ since this would create a cherry. In each of the remaining cases we can find a coloring with $3$ blue vertices, none of which are true blue vertices, and all of the remaining vertices including $v$ will be colored white. If $v$ is adjacent to $v_1$ then we can color $v_1$,$v_4$, and $v_6$ blue, if $v$ is adjacent to $v_2$ then $v_2$, $v_3$, and $v_5$ can be colored blue. If $v$ is adjacent to $v_4$ then $v_3$, $v_4,$ and $v_6$ can be colored blue and if $v$ is adjacent to $v_5$ then $v_2$, $v_3,$ and $v_5$ can be colored blue, and if $v$ is adjacent to $v_6$ then $v_2$, $v_3,$ and $v_6$ can be colored blue.

\begin{figure}[h!]
\begin{center}
\begin{tikzpicture}[node distance = 1cm, line width = 0.5pt]
\coordinate (1) at (0,0);
\coordinate (2) at (1,0);
\coordinate (3) at (2.0,0);
\coordinate (4) at (1.5,0.75);
\coordinate (5) at (1.5,-0.75);
\coordinate (6) at (3,0);

\draw (1)--(2);
\draw (2)--(4);
\draw (2)--(5);
\draw (5)--(3);
\draw (4)--(3);
\draw (3)--(6);
\draw (2)--(3);

\node (A) at (0,0.35) {$v_1$};
\node (B) at (0.95,0.35) {$v_2$};
\node (C) at (1.5,1.1) {$v_3$};
\node (D) at (1.5,-1.1) {$v_4$};
\node (E) at (2.05,0.35) {$v_5$};
\node (F) at (3,.35) {$v_6$};

\foreach \point in {1,2,3,4,5,6} \fill (\point) circle (4pt);

\filldraw [white] 
(1.5,-0.75) circle (3pt)
(1.5,0.75) circle (3pt)
(3,0) circle (3pt)
(0,0) circle (3pt)
(1,0) circle (3pt)
(2,0) circle (3pt);

\end{tikzpicture}
\end{center}
\caption{Graph 76}
\end{figure}
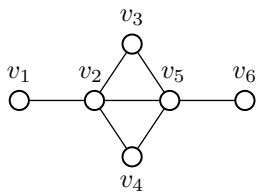

\textbf{Graph 76:} We will show that all but two extensions of $G$ to a graph with 7 vertices will have no true blue vertex of degree 1. Let $v$ denote the new vertex. Note that extensions which would result in a cherry are neglected as those cases are resolved by Lemma 2.6. We have the following cases

\medskip

If $v$ is adjacent to $v_1$ only, we color $v_1$,$v_3$, and $v_4$ blue and $v$, $v_2$,$v_5$ and $v_6$ white. 


\medskip

Note each case above is analogous to another extension of $G$ due to its symmetry. Now we show that the two extensions of $G$ which do contain a true blue vertex of degree 1 may be further extended to remove said true blue vertex. These two extensions are the same by symmetry, and so we consider them as one case. Let our vertices be named as before and include $v_7$ as a vertex of degree 1 adjacent to $v_3$.

\medskip

If $v$ is adjacent to $v_1$ only, we color $v_1$, $v_3$, and $v_4$ blue and $v$,$v_2$,$v_5$,$v_6$ and $v_7$ white.



If $v$ is adjacent to $v_4$ only, we color $v_2$,$v_3$, and $v_4$ blue and $v$,$v_1$,$v_5$,$v_6$ and $v_7$ white.

If $v$ is adjacent to $v_7$, we color $v_2$,$v_5$, and $v_7$ blue and $v$,$v_1$,$v_3$,$v_4$ and $v_6$ white.

\medskip

The remaining cases of connecting $v$ to this extension of $G$ are analogous to one of the previous five. Notice we can always color $v$ white and add edges between $v$ and any subset of the remaining vertices without affecting the number of blue vertices by Lemma 2.4.

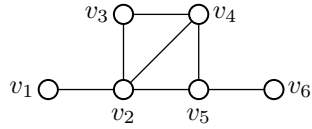
\begin{figure}[h!]
\begin{center}
\begin{tikzpicture}[node distance = 1cm, line width = 0.5pt]
\coordinate (1) at (0,0);
\coordinate (2) at (1,0);
\coordinate (3) at (1,1);
\coordinate (4) at (2,1);
\coordinate (5) at (2,0);
\coordinate (6) at (3,0);

\draw (1)--(2);
\draw (2)--(4);
\draw (2)--(5);
\draw (5)--(4);
\draw (4)--(3);
\draw (5)--(6);
\draw (2)--(3);

\node (A) at (-0.35,0) {$v_1$};
\node (B) at (1,-0.35) {$v_2$};
\node (C) at (0.65,1) {$v_3$};
\node (D) at (2.35,1) {$v_4$};
\node (E) at (2,-0.35) {$v_5$};
\node (F) at (3.35,0) {$v_6$};

\foreach \point in {1,2,3,4,5,6} \fill (\point) circle (4pt);

\filldraw [white] 
(2,0) circle (3pt)
(1,1) circle (3pt)
(3,0) circle (3pt)
(0,0) circle (3pt)
(1,0) circle (3pt)
(2,1) circle (3pt);

\end{tikzpicture}
\end{center}
\caption{Graph 77}
\end{figure}

\textbf{Graph 77:}
 Suppose we add a new vertex called $v$. Due to the symmetry of the graph, it is not necessary to extend on $v_5$ or $v_1$. Our goal with extending is to find no true blue vertex. If we do have a true blue, we must then extend further.

If $v$ is adjacent to $v_6$, then $v_2,v_4,$ and $v_6$ can be filled.

If $v$ is adjacent to $v_3$, then $v_2,v_3,$ and $v_5$ can be filled.

If $v$ is adjacent to $v_1$, then $v_1,v_3,$ and $v_5$ can be filled.

If $v$ is adjacent to $v_4$, then $v_2,v_4,$ and $v$ can be filled. For this case, we must extend further since there is a true blue vertex. Suppose we add a second vertex called $u$.

If $u$ is adjacent to $v_3$, then $v_2,v_3,$ and $v_5$ can be filled.

If $u$ is adjacent to $v$, then $v,v_3,$ and $v_5$ can be filled.

Thus we can extend all cases indefinitely without ever forcing G. Since each case is extendable, then graph 77 is extendable.

 Notice there is no case where $v$ is adjacent to $v_2$ or $v_5$. This is because connecting a vertex to either one would create a cherry structure and that case is covered elsewhere. It is also important to notice that in each case, we are able to find a failed zero-forcing number of three. 
 
 Every blue vertex has two white neighbors, so we can add any number of edges or vertices to any these graphs without forcing G. This is because the three blue vertices will always be adjacent to two white neighbors and cannot force any other vertex. Thus $F(G) \geq 3$ for graph 77.

\begin{figure}[h!]
\begin{center}
\begin{tikzpicture}[node distance = 1cm, line width = 0.5pt]
\coordinate (1) at (0,0);
\coordinate (2) at (1,0);
\coordinate (3) at (2.0,0);
\coordinate (4) at (1.5,0.75);
\coordinate (5) at (1.5,-0.75);
\coordinate (6) at (3,0);

\node (A) at (0,0.4) {$v_1$};
\node (B) at (0.9,0.4) {$v_2$};
\node (C) at (1.5,1.1) {$v_3$};
\node (D) at (1.5,-1.1) {$v_4$};
\node (E) at (2.1,0.4) {$v_5$};
\node (F) at (3,0.4) {$v_6$};

\draw (1)--(2);
\draw (2)--(4);
\draw (2)--(5);
\draw (5)--(3);
\draw (4)--(3);
\draw (3)--(6);
\draw (4)--(5);

\foreach \point in {1,2,3,4,5,6} \fill (\point) circle (4pt);

\filldraw [white] 
(1.5,-0.75) circle (3pt)
(1.5,0.75) circle (3pt)
(3,0) circle (3pt)
(0,0) circle (3pt)
(1,0) circle (3pt)
(2,0) circle (3pt);

\end{tikzpicture}
\end{center}
\caption{Graph 80}
\end{figure}
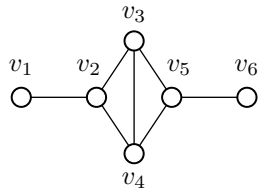

\textbf{Graph 80:} 
Let the vertex on the far left be $v_1$. Let the vertex adjacent to $v_1$ be $v_2$. Let the top vertex adjacent to $v_2$ be $v_3$ and the bottom be $v_4$. Let the vertex that is adjacent to both $v_3$ and $v_4$ on the right be $v_4$. And finally, let the vertex on the far right side be $v_6$. We begin by extending graph 80 by one vertex, $v$ and connecting $v$ to every vertex in $G$. Notice there will be no case where $v$ is adjacent to $v_2$ or $v_5$. This is because connecting a vertex to either would create a cherry structure and that case is covered elsewhere.

If $v$ is adjacent to $v_6$, then $v_3,v_4,$ and $v_6$ can be filled.

If $v$ is adjacent to $v_4$, then $v_2,v_4,$ and $v_5$ can be filled.

If $v$ is adjacent to $v_3$, then $v_2,v_3,$ and $v_5$ can be filled.

If $v$ is adjacent to $v_1$, then $v_1,$ $v_3$, and $v_4$ can be filled.

In all cases there are no true blue vertices. Thus, all cases are extendable and this implies graph 80 can be extended indefinitely. It is also important to notice that in each case, we are able to find a failed zero-forcing number of three.

  \begin{figure}[h!]
\begin{center}
\begin{tikzpicture}[node distance = 1cm, line width = 0.5pt]

\coordinate (1) at (0,1);
\coordinate (2) at (0.2,0);
\coordinate (3) at (0.7,1.6);
\coordinate (4) at (1.4,1);
\coordinate (5) at (1.2,0);
\coordinate (6) at (2.3,1);

\draw (1)--(4);
\draw (2)--(5);
\draw (2)--(1);
\draw (1)--(3);
\draw (4)--(3);
\draw (6)--(4);
\draw (5)--(4);

\node (A) at (2.3,1.3) {$v_1$};
\node (B) at (1.42,1.3) {$v_2$};
\node (C) at (0.7,1.9) {$v_3$};
\node (D) at (0,1.3) {$v_4$};
\node (E) at (-0.15,0) {$v_5$};
\node (F) at (1.55,0) {$v_6$};

\foreach \point in {1,2,3,4,5,6} \fill (\point) 
circle (4pt);

\filldraw [white] 
(0,1) circle (3pt)
(0.2,0) circle (3pt)
(0.7,1.6) circle (3pt)
(1.4,1) circle (3pt)
(2.3,1) circle (3pt)
(1.2,0) circle (3pt);

\end{tikzpicture}
\end{center}
\caption{Graph 85}
\end{figure}
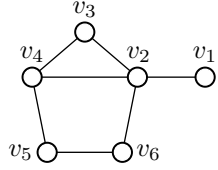

\textbf{Graph 85:} We will show that all but two extensions of $G$ to a graph with 7 vertices have no true blue vertex of degree 1. Let $v$ denote the new vertex. Note that extensions which would result in a cherry are neglected as those cases are resolved by Lemma 2.6. We consider the following cases.

\medskip

If $v$ is adjacent to $v_6$ only, we color $v_2$, $v_6$, and $v_4$ blue and $v$,$v_1$, $v_3$, and $v_5$ white.

If $v$ is adjacent to $v_1$ only, we color $v_1$,$v_3$, and $v_6$ blue and $v$,$v_2$,$v_4$, and $v_5$ white.


If $v$ is adjacent to $v_3$ only, we color $v_2$,$v_3$, and $v_5$ blue and $v$, $v_1$,$v_4$, and $v_6$ white. 

\medskip

Now we show the two extensions of $G$ which do contain a true blue vertex of degree 1 may be further extended to remove said true blue vertex. For the first extension, let our vertices be named as before, and include $v_7$ as a vertex of degree 1 adjacent to $v_5$.

\medskip

If $v$ is adjacent $v_1$ only, we color $v_1$,$v_3$, and $v_5$ blue and $v$,$v_2$,$v_4$,$v_6$, and $v_7$ white. 


If $v$ is adjacent $v_3$ only, we color $v_2$,$v_3$, and $v_5$ blue and $v$, $v_1$,$v_4$,$v_6$ and $v_7$ white. 

If $v$ is adjacent $v_4$ only, we color $v_2$, $v_4$, and $v_5$ blue, and $v$,$v_1$,$v_3$,$v_6$, and $v_7$ white. 


If $v$ is adjacent $v_6$ only, we color $v_2$,$v_4$, and $v_6$ blue and $v$,$v_1$,$v_3$,$v_5$, and $v_7$ white. 

If $v$ is adjacent $v_7$ only, we color $v_2$,$v_4$, and $v_7$ blue and $v$,$v_1$,$v_3$,$v_5$, and $v_6$ white. 

\medskip

For the second extension, let our vertices be named as before, and include $v_7$ as a degree 1 vertex adjacent to $v_4$.

\medskip

If $v$ is adjacent to $v_1$ only, we color $v_1$,$v_4$, and $v_6$ blue and $v$,$v_2$,$v_3$,$v_5$ and $v_7$ white.


If $v$ is adjacent to $v_3$ only, we color $v_2$,$v_3$, and $v_5$ blue and $v$,$v_1$,$v_4$,$v_6$, and $v_7$ white.


If $v$ is adjacent to $v_5$ only, we color $v_2$,$v_4$, and $v_5$ blue and $v$,$v_1$,$v_3$,$v_6$, and $v_7$ white.

If $v$ is adjacent to $v_6$ only, we color $v_2$,$v_4$, and $v_6$ blue and $v$,$v_1$,$v_3$,$v_5$ and $v_7$ white

If $v$ is adjacent to $v_7$ only, we color $v_3$,$v_6$, and $v_7$ blue and $v$,$v_1$,$v_2$,$v_4$ and $v_5$ white.

\medskip

Notice we can always color $v$ white and add edges between $v$ and any subset of the remaining vertices without affecting the number of blue vertices by Lemma 2.4.

\begin{figure}[h!]
\begin{center}
\begin{tikzpicture}[node distance = 1cm, line width = 0.5pt]

\coordinate (1) at (0,0);
\coordinate (2) at (2,0);
\coordinate (3) at (1,1);
\coordinate (4) at (1,2);
\coordinate (5) at (0.5,-1);
\coordinate (6) at (1.5,-1);

\draw (1)--(2);
\draw (1)--(5);
\draw (1)--(3);
\draw (2)--(3);
\draw (2)--(6);
\draw (5)--(6);
\draw (3)--(4);

\foreach \point in {1,2,3,4,5,6} \fill (\point) 
circle (4pt);

\node (A) at (1,2.35) {$v_1$};
\node (B) at (1.35,1) {$v_2$};
\node (C) at (-0.35,0) {$v_3$};
\node (D) at (0.15,-1) {$v_4$};
\node (E) at (1.85,-1) {$v_5$};
\node (F) at (2.35,0) {$v_6$};

\filldraw [white] 
(0,0) circle (3pt)
(2,0) circle (3pt)
(0.5,-1) circle (3pt)
(1,1) circle (3pt)
(1,2) circle (3pt)
(1.5,-1) circle (3pt);

\end{tikzpicture}
\end{center}
\caption{Graph 86}
\end{figure}
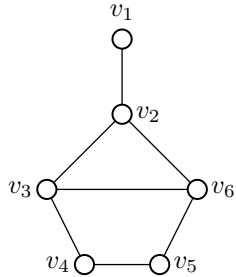

\par\textbf{Graph 86:} We will show that all extensions of $G$ to a graph with 7 vertices have no true blue vertices of degree 1. Let $v$ be the new vertex. Let $v_1$ be the vertex of degree 1, $v_2$ be the vertex of degree 2 that is adjacent to $v_1$, $v_3$ be the vertex of degree 3 that is adjacent to and on the left of $v2$, $v_4$ be the vertex of degree 2 that is adjacent to $v_3$, $v_5$ be the vertex of degree 2 that is adjacent to $v_4$, and $v_6$ be the vertex of degree 3 that is adjacent to $v_5$. We have the following cases.

\medskip

If $v$ is adjacent to $v_1$ only, we color $v_1$,$v_3$, and $v_6$ blue and $v$,$v_2$,$v_4$,$v_5$ white.

If $v$ is adjacent to $v_2$ only, the resulting graph contains a cherry. So we color $v$,$v_1$ white, and color the remaining vertices blue.

If $v$ is adjacent to $v_3$ only, we color $v_2$,$v_3$, and $v_4$ blue and $v$,$v_1$,$v_5$, and $v_6$ white.

If $v$ is adjacent to $v_4$ only, we color $v_2$,$v_4$, and $v_6$ blue and $v$,$v_1$,$v_3$, and $v_5$ white. 

The remaining cases of connecting $v$ to $G$ are analogous to one of the previous four. Also note that in each case where the resulting graph does not contain a cherry, we can always color $v$ white and add edges between $v$ and any subset of the remaining vertices without affecting the number of blue vertices.

\begin{figure}[h!]
\begin{center}
\begin{tikzpicture}[node distance = 1cm, line width = 0.5pt]

\coordinate (1) at (0,0);
\coordinate (2) at (2,0);
\coordinate (3) at (1,1);
\coordinate (4) at (2.5,-1);
\coordinate (5) at (0.5,-1);
\coordinate (6) at (1.5,-1);

\draw (1)--(2);
\draw (1)--(5);
\draw (1)--(3);
\draw (2)--(3);
\draw (2)--(6);
\draw (5)--(6);
\draw (6)--(4);

\foreach \point in {1,2,3,4,5,6} \fill (\point) 
circle (4pt);

\filldraw [white] 
(1,1) circle (3pt)
(2,0) circle (3pt)
(0.5,-1) circle (3pt)
(0,0) circle (3pt)
(2.5,-1) circle (3pt)
(1.5,-1) circle (3pt);

\node (A) at (2.5,-1.4) {$v_1$};
\node (B) at (1.5,-1.4) {$v_2$};
\node (C) at (0.5,-1.4) {$v_3$};
\node (D) at (1,1.4) {$v_4$};
\node (E) at (-0.4,0) {$v_5$};
\node (F) at (2.4,0) {$v_6$};

\end{tikzpicture}
\end{center}
\caption{Graph 87}
\end{figure}
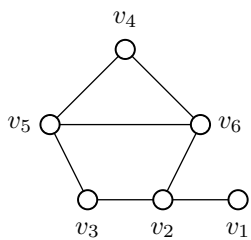

\par\textbf{Graph 87:} All but one of the single vertex extensions of $G$ is either a cherry or suitable for extension by lemma 1.1. Let $v_1$ be the only vertex of degree one that is adjacent to $v_2$; $v_2$ is adjacent to $v_3$ and $v_6$; $v_3$ will be adjacent to $v_2$ and $v_4$ and $v_4$ will be adjacent to $v_5$ and $v_6$.\par If $v$ is adjacent to $v_1$, we can fill in $v_1$, $v_4$, and $v_6$. \par If $v$ is adjacent to $v_2$, it creates a cherry and this case can be discarded. \par If $v$ is adjacent to $v_3$, we can fill in $v_2$, $v_3$, and $v_5$. \par If $v$ is adjacent to $v_5$, we can fill in $v_2$, $v_4$, and $v_5$. \par If $v$ is adjacent to $v_6$, we can fill in $v_2$, $v_4$, and $v_6$.

\par When $v$ is adjacent to $v_4$, it is unavoidable to have a True Blue vertex and must be extended further. Now we will consider a graph in which $v$ is adjacent to $v_4$. We can then add in an additional vertex, $u$. \par If $u$ is adjacent to $v_1$, we can fill in $v_1$, $v_4$, and $v_6$. \par If $u$ is adjacent to $v_2$, it again creates a cherry and we can discard this case. \par If $u$ is adjacent to $v_3$, we can fill in $v_3$, $v_4$, and $v_6$. \par If $u$ is adjacent to $v_4$, we get another cherry and thus another discarded case.\par If $u$ is adjacent to $v_5$, we can fill in $v_2$, $v_4$, and $v_5$. \par If $u$ is adjacent to $v_6$, we can fill in $v_2$, $v_4$, and $v_6$. \par If $u$ is adjacent to $v$, then we can fill $v$, $v_2$, and $v_5$. 

In every case, we see a suitable coloring for extension by Lemma 1.1. Thus we may extend graph 87 to a super-graph $H$, and by Lemma 1.1, $F(H)\geq F(G)\geq 3$.

\begin{figure}[h!]
\begin{center}
\begin{tikzpicture}[node distance = 1cm, line width = 0.5pt]

\coordinate (1) at (0,0);
\coordinate (2) at (0.5,0.5);
\coordinate (3) at (1,1);
\coordinate (4) at (1,1.5);
\coordinate (5) at (1.5,0.5);
\coordinate (6) at (2,0);

\draw (1)--(2);
\draw (2)--(3);
\draw (3)--(4);
\draw (3)--(5);
\draw (5)--(6);
\draw (5)--(2);

\foreach \point in {1,2,3,4,5,6} \fill (\point) 
circle (4pt);

\filldraw [white] 
(1,1) circle (3pt)
(1,1.5) circle (3pt)
(1.5,0.5) circle (3pt)
(2,0) circle (3pt)
(0,0) circle (3pt)
(0.5,0.5) circle (3pt);

\node (A) at (0.6,1) {$v_1$};
\node (B) at (1.9,0.5) {$v_2$};
\node (C) at (0.1,0.5) {$v_3$};
\node (D) at (1,1.9) {$v_4$};
\node (E) at (2,-0.4) {$v_5$};
\node (F) at (0,-0.4) {$v_6$};

\end{tikzpicture}
\end{center}
\caption{Graph 96}
\end{figure}
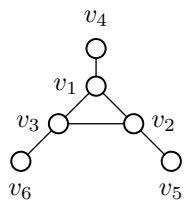

\textbf{Graph 96:} We will show that any extension of $G$ to a graph with 7 vertices will have at most one true blue vertex. Let the vertices of the triangle be $v_{1},v_{2},$ and $v_{3}$ which are adjacent to vertices of degree $1$, $v_{4},v_{5},$ and $v_{6}$ respectively.
	
Consider the addition of a new vertex $v$. We first consider when $v$ is adjacent to one of the vertices $v_{4},v_{5},$or $v_{6}$. Without loss of generality assume $v$ is adjacent to $v_{4}$. Then we color $v_{2},v_{3},$and $v_{4}$ blue and $v,v_{1},v_{5},$ and $v_{6}$ white. Next we consider when $v$ is adjacent to one of the vertices $v_{1},v_{2},$or $v_{3}$. Without loss of generality assume $v$ is adjacent to $v_{1}$. Then we color $v_{1},v_{3}$, and $v_{4}$ blue and $v,v_{2},v_{5},$ and $v_{6}$ white.

\begin{figure}[h!]
\begin{center}
\begin{tikzpicture}[node distance = 1cm, line width = 0.5pt]

\coordinate (1) at (0,0);
\coordinate (2) at (0.5,0);
\coordinate (3) at (1.5,0);
\coordinate (4) at (2,0);
\coordinate (5) at (2.5,0);
\coordinate (6) at (1,0.5);

\draw (1)--(2);
\draw (2)--(3);
\draw (3)--(4);
\draw (5)--(4);
\draw (2)--(6);
\draw (6)--(3);

\foreach \point in {1,2,3,4,5,6} \fill (\point) 
circle (4pt);

\filldraw [white] 
(1.5,0) circle (3pt)
(2.5,0) circle (3pt)
(1,0.5) circle (3pt)
(0,0) circle (3pt)
(0.5,0) circle (3pt)
(2,0) circle (3pt);

\node (A) at (0,-0.4) {$v_1$};
\node (B) at (0.5,-0.4) {$v_2$};
\node (C) at (1,0.9) {$v_3$};
\node (D) at (1.5,-0.4) {$v_4$};
\node (E) at (2,-0.4) {$v_5$};
\node (F) at (2.5,-0.4) {$v_6$};

\end{tikzpicture}
\end{center}
\caption{Graph 98}
\end{figure}
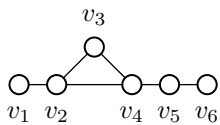

\textbf{Graph 98:} 
We begin by extending graph 98 by one vertex, $v$ and connecting $v$ to every vertex in $G$. Notice there will be no case where $v$ is adjacent to $v_5$, $v_4$, or $v_2$. This is because connecting a vertex to either one would create a cherry and that specific case is covered by the cherry lemma.

If $v$ is adjacent to $v_6$, then $v_2,v_4,$ and $v_6$ can be filled.

If $v$ is adjacent to $v_3$, then $v_2,v_3,$ and $v_5$ can be filled.

If $v$ is adjacent to $v_1$, then $v_1,v_3,$ and $v_5$ can be filled.

We can see in all cases there are no true blue vertices. Thus, all cases are extendable and this implies graph 98 can be extended indefinitely. It is also important to notice that in each case, we are able to find a failed zero-forcing set of size three.

\begin{figure}[h!]
\begin{center}
\begin{tikzpicture}[node distance = 1cm, line width = 0.5pt]

\coordinate (1) at (0,0);
\coordinate (2) at (0,1);
\coordinate (3) at (1,1);
\coordinate (4) at (1,0);
\coordinate (5) at (2,1);
\coordinate (6) at (2,0);

\draw (1)--(2);
\draw (2)--(3);
\draw (3)--(4);
\draw (3)--(5);
\draw (4)--(6);
\draw (1)--(4);

\foreach \point in {1,2,3,4,5,6} \fill (\point) 
circle (4pt);

\filldraw [white] 
(0,1) circle (3pt)
(1,0) circle (3pt)
(2,0) circle (3pt)
(0,0) circle (3pt)
(2,1) circle (3pt)
(1,1) circle (3pt);

\node (A) at (0,-0.4) {$v_1$};
\node (B) at (1,-0.4) {$v_4$};
\node (C) at (2,-0.4) {$v_6$};
\node (D) at (0,1.4) {$v_2$};
\node (E) at (1,1.4) {$v_3$};
\node (F) at (2,1.4) {$v_5$};

\end{tikzpicture}
\end{center}
\caption{Graph 102}
\end{figure}
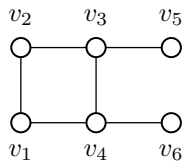

\textbf{Graph 102:} 
We begin by extending graph 102 by one vertex, $v$ and connecting $v$ to every vertex in G. Due to the symmetry of the graph, we can ignore the case where $v$ is adjacent to $v_5$. Notice there will be no case where $v$ is adjacent to $v_3$ or $v_4$. This is because connecting a vertex to either one will create a cherry structure, and this is covered by the cherry lemma.

If $v$ is adjacent to $v_6$, then $v_1,v_3,$ and $v_6$ can be filled.

If $v$ is adjacent to $v_1$, then $v_1,v_3,$ and $v_5$ can be filled. Due to the symmetry of the graph, this case is the same as when $v$ is adjacent to $v_2$. So it can be ignored. But, this case has a true blue and so it must be extended. Suppose we add a new vertex $u$.

If $u$ is adjacent to $v_6$, then $v_1,v_3,$ and $v_6$ can be filled.

If $u$ is adjacent to $v_2$, then $u,v_1,$ and $v_3$ can be filled. Here vertex $u$ is true blue, so we must extend further. Suppose we add another vertex called $w$.

If $w$ is adjacent to $v_6$, then $v_2,v_3,$ and $v_6$ can be filled. Due to the symmetry of the graph, all remaining cases will have the same results. Notice there are no true blue vertices and that $F(G) \geq 3$ in all cases. Since there are no true blue vertices, then this case is now extendable. Since each case is extendable, then graph 102 can be extended indefinitely without being forced.

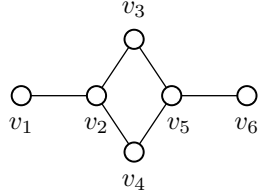
\begin{figure}[h!]
\begin{center}
\begin{tikzpicture}[node distance = 1cm, line width = 0.5pt]
\coordinate (1) at (0,0);
\coordinate (2) at (1,0);
\coordinate (3) at (2.0,0);
\coordinate (4) at (1.5,0.75);
\coordinate (5) at (1.5,-0.75);
\coordinate (6) at (3,0);

\draw (1)--(2);
\draw (2)--(4);
\draw (2)--(5);
\draw (5)--(3);
\draw (4)--(3);
\draw (3)--(6);

\foreach \point in {1,2,3,4,5,6} \fill (\point) circle (4pt);

\filldraw [white] 
(1.5,-0.75) circle (3pt)
(1.5,0.75) circle (3pt)
(3,0) circle (3pt)
(0,0) circle (3pt)
(1,0) circle (3pt)
(2,0) circle (3pt);

\node (A) at (0,-0.4) {$v_1$};
\node (B) at (1,-0.4) {$v_2$};
\node (C) at (1.5,1.15) {$v_3$};
\node (D) at (2.1,-0.4) {$v_5$};
\node (E) at (3,-0.4) {$v_6$};
\node (F) at (1.5,-1.15) {$v_4$};

\end{tikzpicture}
\end{center}
\caption{Graph 103}
\end{figure}

\textbf{Graph 103:}
We begin by extending graph 103 by one vertex, $v$ ,and connecting $v$ to every vertex in G. Notice there will be no case where $v$ is adjacent to $v_5$ or $v_2$. This is because connecting a vertex to either would create a cherry structure and that case is covered by the cherry lemma. It is also important to notice that in each case and sub-case, we are able to find a failed zero-forcing number of three. Our goal with extending is to find no true blue vertex within the graph. Due to the symmetry of the graph, we can ignore the cases where $v$ is adjacent to $v_4$ and $v_1$.

If $v$ is adjacent to $v_6$, then $v_3,v_4,$ and $v_6$ can be filled.

If $v$ is adjacent to $v_3$, then $v_2,v_5,$ and $v_6$ can be filled. Since $v_6$ is true blue, we must then extend further. Suppose we add a new vertex called $u$.

If $u$ is adjacent to $v_6$, then $v_3,v_4,$ and $v_6$ can be filled.

If $u$ is adjacent to $v$, then $v,v_2,$ and $v_5$ can be filled.

If $u$ is adjacent to $v_4$, then $v_4,v_5,$ and $v_6$ can be filled. Here, $v_6$ is a true blue vertex, so we must extend to 9 vertices. Suppose we add a new vertex called $w$.

If $w$ is adjacent to $v_1$, then $v_1,v_3,$ and $v_4$ can be filled.

If $w$ is adjacent to $v_6$, then $v_3,v_4,$ and $v_6$ can be filled.

If $w$ is adjacent to $v$, then $v,v_2,$ and $v_5$ can be filled.

If $w$ is adjacent to $u$, then $v,v_2,$ and $v_5$ can be filled.

Since there are no more true blue vertices, then this case is resolved and can be extended indefinitely. Since each case is extendable, this implies that graph 103 can also be extended indefinitely without being forced..

\begin{figure}[h!]
\begin{center}
\begin{tikzpicture}[node distance = 1cm, line width = 0.5pt]

\coordinate (1) at (0,0);
\coordinate (2) at (2,0);
\coordinate (3) at (1,0.75);
\coordinate (4) at (1,1.5);
\coordinate (5) at (0.5,-1);
\coordinate (6) at (1.5,-1);

\draw (1)--(5);
\draw (1)--(3);
\draw (2)--(3);
\draw (2)--(6);
\draw (5)--(6);
\draw (3)--(4);

\foreach \point in {1,2,3,4,5,6} \fill (\point) 
circle (4pt);

\node (A) at (1.4,1.5) {$v_1$};
\node (B) at (1.4,0.75) {$v_2$};
\node (C) at (-0.4,0) {$v_3$};
\node (D) at (0.5,-1.4) {$v_4$};
\node (E) at (1.4,-1.4) {$v_5$};
\node (F) at (2.4,0) {$v_6$};

\filldraw [white] 
(0,0) circle (3pt)
(2,0) circle (3pt)
(0.5,-1) circle (3pt)
(1,0.75) circle (3pt)
(1,1.5) circle (3pt)
(1.5,-1) circle (3pt);

\end{tikzpicture}
\end{center}
\caption{Graph 105}
\end{figure}
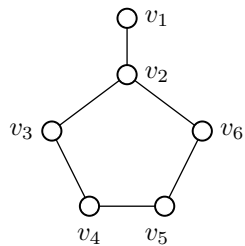

\textbf{Graph 105:}
We begin by extending graph 105 by one vertex, $v$ and connecting $v$ to every vertex in G. Notice there will be no case where $v$ is adjacent to $v_2$. This is because adding an extension to $v_2$ would create a cherry structure and this case is covered by the cherry lemma. It is also important to notice that in each case, we are able to find a failed zero-forcing number of three. Our goal with extending is to find no true blue vertex within the graph.

If $v$ is adjacent to $v_1$, then $v_1,v_3,$ and $v_6$ can be filled.

If $v$ is adjacent to $v_6$, then $v_2,v_4,$ and $v_6$ can be filled.

If $v$ is adjacent to $v_3$, then $v_2,v_3,$ and $v_5$ can be filled.

If $v$ is adjacent to $v_5$, then $v_1,v_2,$ and $v_5$ can be filled.

If $v$ is adjacent to $v_4$, then $v_1,v_2,$ and $v_4$ can be filled.

Notice the last two cases have a true blue vertex. Due to the symmetry of the graph, we can focus on one and ignore the other. Assume the case where $v$ is adjacent to $v_5$. Suppose we add a new vertex called $u$.

If $u$ is adjacent to $v_4$, then $v_2,v_4,$ and $v_5$ can be filled.

If $u$ is adjacent to $v_6$, then $v_2,v_4,$ and $v_6$ can be filled.

If $u$ is adjacent to $v_3$, then $v_2,v_3,$ and $v_5$ can be filled.

If $u$ is adjacent to $v$, then $v,v_2,$ and $v_4$ can be filled.

If $u$ is adjacent to $v_1$, then $v_1,v_3,$ and $v_6$ can be filled.

Since there are no more true blue vertices in these sub-cases, then this case is resolved and can be extended indefinitely. Since each case is extendable, this implies that graph 105 can be extended indefinitely without being forced and $F(G) \geq 3$.

\medskip

\textbf{Graph 112:} 

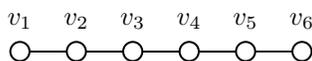
\begin{figure}[h!]
\begin{center}
\begin{tikzpicture}[node distance = 1cm, line width = 0.5pt]
\coordinate (1) at (0,0);
\coordinate (2) at (0.75,0);
\coordinate (3) at (1.5,0);
\coordinate (4) at (2.25,0);
\coordinate (5) at (3,0);
\coordinate (6) at (3.75,0);

\draw \foreach \x [remember=\x as \lastx (initially 1)] in {2,3,4,5,6}{(\lastx) -- (\x)};
\draw (6)--(4);
\draw (5)--(1);
\draw (5)--(4);
\draw (5)--(3);
\draw (5)--(2);
\draw (1)--(2);
\draw (2)--(3);
\draw (3)--(4);

\node (A) at (0,.4) {$v_1$};
\node (A) at (0.75,.4) {$v_2$};
\node (A) at (1.5,.4) {$v_3$};
\node (A) at (2.25,.4) {$v_4$};
\node (A) at (3,.4) {$v_5$};
\node (A) at (3.75,.4) {$v_6$};



\foreach \point in {1,2,3,4,5,6} \fill (\point) circle (4pt);

\filldraw [white] 
(0,0) circle (3pt);
\filldraw [white] 
(0.75,0) circle (3pt);
\filldraw [white] 
(1.5,0) circle (3pt);
\filldraw [white] 
(2.25,0) circle (3pt);
\filldraw [white] 
(3,0) circle (3pt);
\filldraw [white] 
(3.75,0) circle (3pt);
\end{tikzpicture}
\end{center}
\caption{Graph 112}
\end{figure}
$P_{6}$: Any extension of $P_{6}$ to a graph with 7 vertices will have at most one true blue vertex.
	
Let the vertices of $v_{1}$ or $v_{6}$ then we can color vertex $v$ white and color vertices $v_{1},v_{3},$and $v_{5}$ blue to create a failed zero-forcing set of size $3$. If $v$ is not adjacent to either $v_{1}$ or $v_{6}$ then $v$ must be adjacent to one or more of the vertices $v_{2},v_{3},v_{4},$ or $v_{5}$. We consider different cases, each showing that we can create a failed zero-forcing set of size $3$ with at most one true blue vertex. We consider a series of different cases.

\medskip 

If $v$ is adjacent to $v_{2}$ and $v_{3}$ only then we color $v_{2},v_{3},$ and $v_{5}$ blue and $v,v_{1},v_{4},$ and $v_{6}$ white.
	
If $v$ is adjacent to $v_{2}$ and $v_{4}$ only then we color $v_{2},v_{3},$ and $v_{4}$ blue and $v,v_{1},v_{5},$ and $v_{6}$ white.
	
If $v$ is adjacent to $v_{2}$ and $v_{5}$ only then we color $v,v_{2}$ and $v_{5}$ blue and $v_{1},v_{3},v_{4},$ and $v_{6}$ white.
	
If $v$ is adjacent to $v_{3}$ and $v_{4}$ only then we color $v,v_{2},$ and $v_{5}$ blue and $v_{1},v_{3},v_{4},$ and $v_{6}$ white.
	
If $v$ is adjacent to $v_{2},v_{3},$ and $v_{4}$ only then we color $v_{2},v_{3}$, and $v_{5}$ blue and $v_{1},v_{3},v_{4},$ and $v_{6}$ white.
	
If $v$ is adjacent to $v_{2},v_{3},$ and $v_{5}$ only then we color $v,v_{2},$ and $v_{4}$ blue and $v_{1},v_{3},v_{5},$ and $v_{6}$ white.
	
If $v$ is adjacent to $v_{2},v_{3},v_{4}$ and $v_{5}$ only, then we color $v_{2},v_{3}$, and $v_{5}$ blue and $v,v_{1},v_{4}$, and $v_{6}$ white.

\begin{subsection}{Proof of the main theorems} 
We now prove Theorem 2.1. Let $G$ be a graph with at least $7$ vertices. Then $F(G^{\prime })\geq 3$.

\begin{proof}
We will show that all graphs with $8$ vertices have a failed zero-forcing set of size $3$ and have at most one true blue vertex. Now start with a graph $G$ with $n=9$ or more vertices. Remove $n-8$ vertices along with their incident edges to create the graph $G^{\prime }$. By our previous work, $F(G^{\prime })\geq 3$ and $G^{\prime }$ has no true blue vertices. By the Extension Lemma 2.4, we can conclude that $F(G)\geq 3$.
\end{proof}	
\end{subsection}

We have now proved our main theorem that gives us a characterization of all graphs with $F(G)=2$.

\begin{theorem}
The graphs with $F(G)=2$ are precisely those shown in Figure 1.
\end{theorem}

\section{Acknowledgements}  This research was supported by the National Science Foundation Research for Undergraduates Award 1950189.



\end{document}